\documentclass[graybox]{svmult}

\usepackage{type1cm}        % activate if the above 3 fonts are
\usepackage{makeidx}         % allows index generation
\usepackage{graphicx}        % standard LaTeX graphics tool
\usepackage{multicol}        % used for the two-column index
\usepackage[bottom]{footmisc}% places footnotes at page bottom

\usepackage{newtxtext}       % 
\usepackage{newtxmath}       % selects Times Roman as basic font

\newcommand{\pr}{\mathbb{P}}
\newcommand{\bE}{\mathbb{E}}

\newcommand{\n}{\mathbb{N}}

\newcommand{\mR}{\mathcal{R}}
\newcommand{\mB}{\mathcal{B}}

\newcommand{\1}{{\bf 1}}

\makeindex             % used for the subject index
\begin{document}

\title*{Limit theorem for reflected  random walks}
\author{Hoang-Long Ngo and Marc Peign\'e}
\institute{Hoang-Long Ngo \at Hanoi National University of Education, 136 Xuan Thuy - Cau Giay - Hanoi - Vietnam, \email{ngolong@hnue.edu.vn}
\and Marc Peign\'e \at Institut Denis Poisson UMR 7013,  Universit\'e de Tours, Universit\'e d'Orl\'eans, CNRS  France \email{peigne@univ-tours.fr}}
\maketitle

\abstract*{Let $\xi_n, n \in \mathbb N$  be a sequence of i.i.d.  random variables with values in $\mathbb Z$.
	The associated random walk  on $\mathbb Z$ is $S(n)= \xi_1 +   \cdots +\xi_{n+1}$ and the corresponding
	``reflected walk''    on  $\mathbb N_0$ is  the Markov chain $X=(X(n))_{n\geq 0}$ given by  $X(0) = x \in \mathbb N_0$ and
	$X(n+1) = \vert X(n) + \xi_{n+1}\vert $  for  $n \geq 0$. It is well know that the reflected walk $(X(n))_{n \geq 0}$ is  null-recurrent   when the  $\xi_n$ are square integrable and centered. In this paper, we prove that the  process $(X(n))_{n \geq 0}$, properly rescaled, converges in distribution towards the reflected  Brownian motion on $\mathbb R^+$, when  $\mathbb E[\xi_n^2]<+\infty,   \bE[(\xi_n^-)^{3 }] < + \infty$  and the $\xi_n$ are aperiodic and centered.}

\abstract{Let $\xi_n, n \in \mathbb N$  be a sequence of i.i.d.  random variables with values in $\mathbb Z$.
	The associated random walk  on $\mathbb Z$ is $S(n)= \xi_1 +   \cdots +\xi_{n+1}$ and the corresponding
	``reflected walk''    on  $\mathbb N_0$ is  the Markov chain $X=(X(n))_{n\geq 0}$ given by  $X(0) = x \in \mathbb N_0$ and
	$X(n+1) = \vert X(n) + \xi_{n+1}\vert $  for  $n \geq 0$. It is well know that the reflected walk $(X(n))_{n \geq 0}$ is  null-recurrent   when the  $\xi_n$ are square integrable and centered. In this paper, we prove that the  process $(X(n))_{n \geq 0}$, properly rescaled, converges in distribution towards the reflected  Brownian motion on $\mathbb R^+$, when  $\mathbb E[\xi_n^2]<+\infty,   \bE[(\xi_n^-)^{3 }] < + \infty$  and the $\xi_n$ are aperiodic and centered. }

\section{Introduction and Notations}
Let $(\xi_n)_{n\geq 1}$ be a sequence of $\mathbb Z$-valued, independent and identically distributed random variables,  with common law $\mu$ defined on a probability space $(\Omega, \mathcal{F}, \pr)$.
We denote $S=(S(n))_{n \geq 0}$ the classical random walks with steps $\xi_k$ defined by $S(0)=0$ and $S(n)= \xi_1+\ldots +\xi _n$ for any $n \geq 1$.

Throughout  this paper, we denote $\mathbb N_0$ the set of non-negative integers and we consider the   {\it  reflected random walk } $(X(n))_{n\geq 0}$  on $\mathbb N_0$ defined by 
\begin{equation*} \label{defX}
X(n+1) = |X(n) + \xi_{n+1}|, \quad \text{ for } n \geq 0,
\end{equation*} 
where $X(0)$ is a $\mathbb{N}_0$-valued random variables. 
When $X(0)= x$ $\pr$-a.s., with $x \in \mathbb{N}_0$, the process $(X(n))_{n\geq 0}$ is also denoted {by}  $(X^x(n))_{n\geq 0}$.   It   evolves  as the random walk $x+S(n)$ as long as it stays non negative. When $x+S(n)$ enters the set of negative integers, the sign of its value is changed; the same construction thus  applies starting from $\vert x+S(n)\vert, \ldots $  and so on.  

The process $(X^x(n))_{n\geq 0}$ is a Markov chain on $\mathbb{N}_0$ starting from $x$. 
Several papers describing its stochastic behavior have been published; we refer to \cite{PW} where the recurrence of the reflected random walk is studied under some conditions which are nearly to be optimal. The reader may find also several references therein.  

Firstly, $(X^x(n))_{n\geq 0}$ has some similarities with the classical random walk on $\mathbb R$; for instance, a strong law of large numbers holds, namely 
$$
\lim_{n \to +\infty} {X^x(n)\over n}=0 \quad \mathbb P\text{-a.s.}
$$
when $\mathbb E[\vert \xi_n\vert]<+\infty$ and $\mathbb E[  \xi_n\ ]=0$ (see Lemma \ref{SLLN} in section  3).  Nevertheless, in contrast to what holds for the classical random walk on $\mathbb R$,  this does not yield to the recurrence of $(X^x(n))_{n\geq 0}$. In  \cite{PW}, it is proved that  the process $(X^x(n))_{n \geq 0}$ is null-recurrent when  
$\mathbb E [\vert \xi_n\vert ^{3/2}]<+\infty$ and $\mathbb E [  \xi_n]=0$ and that $(X^x(n))_{n \geq 0}$ may be transient  when $\mathbb E [\vert \xi_n\vert ^{{3/2}}]=+\infty$, even if   $\mathbb E [\vert \xi_n\vert ^{{3/2}-\epsilon}]<+\infty$ for any $\epsilon >0$.
The reader can find in  \cite{Kemperman}   a necessary and sufficient condition for the recurrence  of $(X^x(n))_{n\geq 0}$  (see Theorem 4.6) but this condition  cannot be reduced to the existence of some moments.

Once the strong law of large number holds, it is natural to study  the oscillations of the process around its expectation. Let us state our result.

\begin{theorem}\label{meander}
	Let $(\xi_n)_{n \geq 1}$ be  a sequence of $\ \mathbb Z$-valued   i.i.d. random variables such that
	
	\begin{enumerate} 
		\item [{\bf A1.}]  	$\bE[\xi_n^2]=\sigma^2 <+ \infty$ and  $\ \bE[ (\xi_n^-)^{3 }] <+ \infty \quad ^(\footnote{ $\xi_n^-=\max(0, -\xi_n)$ denotes the negative part of $\xi_n$}^)$;
		
		\item [{\bf A2.}]	$\bE[\xi_n] = 0$; 
		
		\item [{\bf A3.}] The distribution of the $\xi_n$ is strongly aperiodic, i.e. the support of the distribution of  $\xi_n$ is not included in the coset of a proper subgroup of $\mathbb Z$.	
	\end{enumerate}
	Let $(X(t) )_{t\geq0}$ be the continuous time process   constructed  from the sequence $(X(n))_{n \geq 0}$ by  linear interpolation between the values at integer points. Then,  as $n \to +\infty$,  the sequence of stochastic processes $(X_n(t))_{n \geq 1}$, defined by
	
	\[
	X_n(t):= {1\over \sigma \sqrt{n}}X(nt), \quad n \geq 1, 0\leq t\leq 1, 
	\]
	weakly converges  in the space of continuous functions  on $[0, 1]$ to the absolute value $(\vert B(t)\vert)_{t \geq 0}$ of the Brownian motion on $\mathbb R$.
\end{theorem}

Let us insist on the fact that   $ X^x(n)$ coincides  with   $x+S(n) $  as long as it stays non-negative, but after it may    differ  drastically. The sequence   of successive reflection  times of $ (X^x(n))_{n \geq 0}$ introduces  some strong  inhomogeneity on time and makes it necessary to adopt a totally different approach to prove  an invariance principle as stated above.

A model which is quite similar to $(X^n(x))_{n \geq 0}$ is the queuing process $(W^x(n))_{n \geq 0}$, also called the {\it Lindley process}, corresponding to the waiting times in a single server queue. We think to $(W^x(n))_{n \geq 0}$ as an absorbing random walk on $\mathbb N_0$; as $W^x(n)$, it evolves as the random walk $x+S(n)$ as long as it stays non-negative and, when it attempts to cross $0$ and become negative, the new value is reset to $0$ before continuing.  We refer to \cite{MarcLong}  for  precise  descriptions and variations on this process and follow  the same  strategy  to obtain the invariance principle.

The  excursions of $(W^x(n) )_{n \geq 0}$ and $(X^x(n))_{n \geq 0}$  between  two consecutively  times  of absorption-reflection
coincide with   some parts of the trajectory of  $(S(n))_{n \geq 0}$, up to a translation; thus, their study is related to   the fluctuations of $(S(n))_{n \geq 0}$. Hence, as in \cite{MarcLong},  we   introduce the sequence of strictly descending ladder epochs $(\ell_l )_{l \geq 0}$  of the random walk $ (S(n))_{n \geq 0}$ defined inductively by $\ell_0 =0$ and, for any $l \geq 1$,
\[
\ell_{l+1}  := \min\{n>\ell_l \mid S(n) < S{(\ell_l) }\}.
\]
When $\mathbb E[\vert \xi_n\vert ] <+\infty$ and $\mathbb E[ \xi_n ] =0, $ the random variables $\ell_{1}  ,\ell_{2}   -\ell_{1}, \ell_{3}  -\ell_{2},\ldots  $ are $\mathbb P$-a.s. finite and i.i.d. and  the same property holds for the random variables
$S{(\ell_{1})  }, S{(\ell_{2} ) }-S{(\ell_{1})  }, S{(\ell_{3})  }-S{(\ell_{2})},\ldots $. In other words,  the processes  $(\ell_l )_{l\geq 0}$ and $(S{(\ell_l )})_{l \geq 0}$  are random walks on $\mathbb N_0$  and  $\mathbb Z$  with respective distribution $\mathcal L(\ell_1)$  and $\mathcal L(S{(\ell_1) }  )$.

Let us briefly point out  the main difference between   $(W^x(n) )_{n \geq 0}$ and $(X(n))_{n \geq 0}$. At an absorption time, the value of the process 
$ W^x(n) $ is reset to $0$ before continuing as a classical random walk for a while: there is a total loss of memory of the past  after each absorption.  Rather,  at a reflection  time, the process $X^x(n)$ equals the absolute value of $x+S(n)$. This value   is the ``new'' starting point  of the process, for a while, and    has a great influence on the  next reflection time; in other words, the process always captures   some memory of the past at any time of reflection. This phenomenon has to be taken into account and requires a precise study of the sub-process $(X(r_k))_{k \geq 0} $ of   $(X(n))_{n \geq 0}$ corresponding to these successive  times  $(r_k)_{k \geq 0}$ of reflection; our strategy  consists in studying the spectrum  of the transition probabilities matrix $\mathcal R$ of $(X(r_k))_{k \geq 0}$, acting  on some   Banach space 
$\mathcal B=\mathcal B_\alpha $ of   functions from $\mathbb N_0$ to $\mathbb C$  with  
growth less than $x^\alpha$ at infinity, for some  $\alpha >0$  to be fixed. 
In particular, in order to apply  recent results on renewal sequences  \cite{Gouezel}, we need   precise estimates on the tail of distribution of the reflection times; this is the main reason  of the restrictive assumption  $\bE[ (\xi_n^-)^3] <+ \infty$ instead of moment of order 2, as we could expect.  More precisely,  throughout the paper, we need the following  properties to be satisfied:

\begin{enumerate}
	\item [(i)] The operator $\mathcal R$ acts on $\mathcal B_\alpha$. 
	
	This holds when $\mathbb E[\vert S(\ell_1)\vert ^{1+\alpha}]<+\infty$ and yields to the condition $\bE[ (\xi_n^-)^{2+\alpha}] <+ \infty$ (see Proposition \ref{propspectral}).
	\item [(ii)] The function $\mathbb N_0\to \mathbb N_0, x \mapsto x, $  belongs to $\mathcal B_\alpha$; this imposes the condition  $\alpha \geq 1$ (see Proposition \ref{renewalsequence}).
\end{enumerate}
Eventually, we  fix  $\alpha =1$ from Section \ref{meander} on.

\noindent {\bf Notations.}  {\it Throughout  the text, we use the following notations.
	Let $u=(u_n)_{n \geq 0}$ and $v=(v_n)_{n \geq 0}$ be two sequences of positive reals; we write
	
	$\bullet\quad $ {\it
		$u\stackrel{c}{\preceq}v $} (or simply $u \preceq v $) when $u_n \leq c v_n$ for some constant $c>0$ and $n$ large enough;
	
	$\bullet\quad $ $u_n \sim v_n$ when $\lim_{n\to+ \infty} \frac{u_n}{v_n} = 1$.  
	
	$\bullet\quad $ $u_n \approx v_n$ when $\lim_{n\to+ \infty} (u_n - v_n) = 0$.   
	
}

\noindent {\bf Acknowledgment}
H.-L. Ngo thanks the University of Tours for generous hospitality in the Instittue Denis
Poisson (IDP) and financial support in May 2019.  This article is a result of the research team with the title "Quantitative Research Methods  in Economics and Finance", Foreign Trade University, Ha Noi, Vietnam.

M. Peign\'e thanks the Vietnam Institute
for Advanced Studies in Mathematics (VIASM) and the Vietnam Academy of Sciences And
Technology (VAST) in Ha Noi for their kind and friendly hospitality and accommodation in
June 2018.

Thanks  are also due to M. Pollicott who proposed to publish this article in the Chaire Jean Morlet Series.

Both authors thank the referee for many helpful comments that improved the text and some proofs. Duy Tran Vo also pointed them several misprints.

\section{Fluctuations of random walks and auxiliary estimates}

\subsection{On the fluctuation of random walks} \label{sec:2.1} 

Let $h$ be the Green function of the random walk $ (S{(\ell_l) } )_{l \geq 0}$, called sometimes  the ``descending renewal function''   of $S$, defined by
\begin{equation*} \label{def:h}
h(x) = \begin{cases} \displaystyle \sum_{l=0}^{+\infty} \pr[S(\ell_l)\geq -x] & \text{if } \quad x \geq 0,\\
0 & \text{otherwise.}\end{cases}
\end{equation*} 

The function  $h$ is  harmonic for the random walk $(S(n))_{n \geq 0}$  killed when it reaches  the negative half line $(-\infty; 0]$; namely, for any $x \geq 0,$
\[
\bE[h(x+\xi_1); x+\xi_1>0  ] = h(x).
\]
This holds for any oscillating random walk, possible without finite second moment. 

Similarly, we denote $\tilde h$ the ascending renewal function of the  random walk $ (S{(n) } )_{n \geq 0}$ (i.e the descending renewal function  of  $ (-S{(n) } )_{n \geq 0}$). 

Both functions $h$ and   $\tilde h$  are  increasing, $h(0)=\tilde h(0)=1$ and  $h(x) = O(x), \tilde h(x) = O(x)$ as $x \to+ \infty$ (see \cite{AGKV05}, p. 648).

We have also to take into account the fact  that  the random walk $S$ does not always start from  the origin; hence,  for any $x \geq 0$, we set $\tau^S(x):= \inf\{ n \geq 1: x+S(n) < 0\}$;  it holds 
\[
[\tau^S(x)>n]= [L_n \geq  -x],
\]
where    $L_n=\min(S(1), \ldots, S(n))$.
The following  result is a combination of   Theorem 2 and  Proposition 11 in \cite{Doney12}  and   Theorem A in \cite{Kozlov}  
(see also Theorems II.6  and   II.7 in \cite{LePagePeigne1}).
\begin{lemma} \label{LemA}
	For any $x  \geq 0$, 
	\begin{enumerate}
		\item
		\begin{equation*}\label{LemA1}
		\pr[\tau^S(x)> n]  \sim  c_1 \frac{h(x)}{\sqrt{n }} \qquad as \quad n \to +\infty, 
		\end{equation*}
		where   $  \displaystyle  c_1=
		{\mathbb E[-S_{\ell_1}]\over \sigma\sqrt{2\pi}}$.
		Moreover,  
		there exists  a constant  $C_1 > 0$ such that for any $x \geq 0$ and $n \geq 1,$
		\begin{equation*}\label{LemA2}
		\pr[\tau^S(x)> n] \leq C_1\frac{h(x)}{\sqrt{n}}.
		\end{equation*}
		\item  For any $x, y \geq 0$, 
		\begin{equation*}\label{LemA'1}   
		\pr[\tau^S(x)> n, x+S(n) = y] \sim {1\over \sigma\sqrt{2\pi}} \frac{h(x)\tilde h(y)}{n^{{3/2}}} \qquad {\it as} \quad n \to +\infty, 
		\end{equation*}
		and there exists  a  constant   $ C_2>0$ such that,  for any any $x, y \geq 0$ and $n \geq 1,$
		\begin{equation*}\label{LemA'2}
		\pr[\tau^S(x)> n,  x+S(n)=y] \leq C_2\frac{h(x)\tilde h(y)}{n^{{3/2}}}.
		\end{equation*}
	\end{enumerate}
\end{lemma}
These assertions  yield   a precise estimate of the probability $\pr[\tau^S(x) = n]$ itself, and not only the tail  of the distribution of $\tau^S$. As a direct consequence, the sequence of descending ladder epochs  $(\ell_l)_{l \geq 1}$ of the random walk $(S(n))_{n \geq 0} $ satisfies some renewal theorem  \cite{Doney12}.  Let us   state these two consequences which enlighten the next section where similar  statements concerning the successive epochs of reflections of the reflected random are proved.

\begin{corollary}  \label{cororenewal}
	For any $x  \geq 0$,   
	$$ \pr[\tau^S(x) = n] \sim  {c_1 \over 2 }   \ h(x)\  {1\over  n^{{3/2}}}  \qquad {\it as} \quad n \to +\infty, $$
	and  there exists   a constant   $\     C_3>0$ such that,  for any  $x \geq 0$ and $n\geq 1$,
	\[
	\pr[\tau^S(x) = n] \leq  C_3\frac{h({x})}{n^{{3/2}}}.
	\]
	Furthermore,
	\[
	\sum_{l=0}^{+\infty} \pr[\ell_l = n] \sim   \frac{1}{c_1 \pi}\ {1\over  \sqrt{n}}
	\qquad {\it as} \quad n \to +\infty. 
	\]
\end{corollary}

\subsection{Conditional    limit theorems}

The following statement  corresponds to Lemma 2.3 in \cite{AGKV05}; the symbol $``\Rightarrow"$ means ``weak convergence".
% We denote by $(S_t)_{t\geq 0}$   the continuous time process  constructed  from the sequence $(S(n))_{n\geq 0}$ by using the linear interpolation between the values at integer points. 
%
\begin{lemma} \label{LemC2}Assume $\mathbb E(\xi_i^2)<+\infty$ and $\mathbb E(\xi_i)=0$. Then, for any $x \geq 0$, 
	$$\mathcal{L}\left(\Big(\frac{S([nt])}{\sigma\sqrt{n}}\Big)_{0\leq t \leq 1}| \min\{S(1), \ldots, S(n)\}\geq -x\right) \Rightarrow \mathcal{L}(L^+) \quad \text{as } n \to +\infty,$$
	where $L^+$ is the Brownian meander. 
	
	In particular,  for  any 
	%	$$\mathcal{L}\left(\Big(\frac{S([nt])}{\sigma\sqrt{n}}\Big)_{0\leq t \leq 1}| \min\{S(1), \ldots, S(n)\}\geq -x\right) \Rightarrow \mathcal{L}(L^+) \quad \text{as } n \to +\infty,$$
	%	where $L^+$ is the Brownian meander. 
	%In particular, for any   
	bounded and  Lipschitz continuous function $\phi: \mathbb R \to \mathbb R$,
	$$\lim_{n\to +\infty} \bE\left[ \phi \left( \frac{x+ S(n)}{\sigma\sqrt{n}}\right)\Big| \tau^S(x) > n\right] = \int_0^{+\infty} \phi(z)ze^{-z^2/2}dz.$$
\end{lemma}
This Lemma is useful in the sequel to control the fluctuations of the excursions of the process $(X(n))_{n \geq 0}$ between two  
successive  times of reflection. In order to control also the higher  dimensional  distributions of these excursions, we need    some
invariance principle for random walk bridges conditioned to stay positive. The following result corresponds in our setting to Corollary 2.5   in \cite{CC}.

\begin{lemma} \label{lemCC} 
	For any   bounded, Lipschitz continuous function $\phi: \mathbb R \to \mathbb R$, any $x, y \geq 0$, and any $t > s > 0$, 
	\begin{align*}
	\lim_{n\to +\infty} &\bE\left[ \phi \left( \frac{x+ S([ns])}{\sigma\sqrt{n}}\right)\Big| \tau^S(x) > [nt], x+S([nt])=y\right] \\
	&= \int_0^{+\infty} 2\phi(u\sqrt{s}) \exp \left( -\frac{u^2}{2\frac st \frac{t-s}{t}}\right) \frac{u^2}{\sqrt{2\pi \frac{s^3}{t^3}\frac{(t-s)^3}{t^3}}}du.
	\end{align*}
\end{lemma}

%%%%%%%%%%%%%%%%%%%%%%%%%%%%%
\section{On the sub-process of reflections} 

We present briefly some results from \cite{EP} and \cite{PW}. The  reflected times $ r_n, n \geq 0,$ of the random walk $(X(n))_{n \geq 0}$ are defined by: for any $x \geq 0$, 
$$r_0=r_0(x) = 0\quad \text{ and} \quad  r_{n+1}  = \inf \{m> r_n  \mid   X(r_n) + \xi_{r_n+1}+\cdots+ \xi_m <  0\}.$$
Notice that these random variables are  $\mathbb N_0\cup\{+\infty\}$-valued  stopping times with respect to the filtration $(\mathcal G_n)_{n \geq 0}$.

When $\mathbb E[\vert \xi_n\vert]<+\infty$ and $\mathbb E[ \xi_n ]=0$, the random walk $(S(n))_{n \geq 0}$  is oscillating, hence the   $r_n, n \geq 0$, are all finite $\mathbb P$-a.s.
and $ S(n)/n $  converges $\mathbb P$-a.s. towards $0$. The strong law of large numbers is still true for the reflected random walk $(X^x(n))_{n \geq 0}$ on $\mathbb N_0$ but  does not derive directly.   
\begin{lemma} \label{SLLN}
	If $\ \mathbb E[\vert \xi_n\vert]<+\infty$ and $\mathbb E[ \xi_n ]=0$, then, for any $x \in \mathbb N_0$, 
	\[
	\lim_{n \to +\infty}{X^x(n)\over n} =0 \quad \mathbb P\text{-a.s.}
	\]
\end{lemma}
\begin{proof}
	For any $n\geq 1$, there exists a   (random) integer  $k_n\geq 1$ such that  $r_{k_n}\leq n<r_{k_n+1}$. It holds \[
	X^x(n)= X^x(r_{k_n})+\left( \xi_{r_{k_n}+1}+\cdots+ \xi_n\right)= X^x(r_{k_n})+ S(n)-S(r_{k_n}),
	\] so that
	$$
	0\leq {X^x(n)\over n}= {X^x(r_{k_n})\over n}+{S(n)\over n}-{S(r_{k_n})\over n}\leq 
	{\max\{\vert \xi_1\vert, \ldots,  \vert \xi_n\vert\}\over n}+{S(n)\over n}-{S(r_{k_n})\over n}.
	$$
	
	The first term on the right hand side converges $\mathbb P$-a.s.  towards $0$  since $ \mathbb E[\vert \xi_n\vert]<+\infty$.
	
	By the strong law of large number, the second term tends $\mathbb P$-a.s.  to $0$. 
	
	At last, the same property holds   for  the last term,  since 
	$\displaystyle \Bigl\vert {S(r_{k_n})\over n}\Bigr\vert= \Bigl\vert {S(r_{k_n})\over r_{k_n}}\Bigr\vert\times   { r_{k_n}\over n}  \leq \Bigl\vert {S(r_{k_n})\over r_{k_n}}\Bigr\vert$.
\end{proof}
It follows from Lemma 2.3 in \cite{PW06}  that the sub-process of reflections  $(X(r_k))_{k\geq 0}$ is a Markov chain on $\mathbb{N}_0$ with transition probability $\mR$ given by: for all $x,y \in \mathbb N_0$, 
\begin{equation}\label{matrixR}
\mR(x,y) = \begin{cases} 0 & \text{ if } y = 0\\
\sum_{w= 0}^x U^*(-w)\mu^*(w-x-y) & \text{ if } y \geq 1,  \end{cases}  	\end{equation}
where  $\mu^*$ is the distribution of $S(\ell_1)$  and $\displaystyle U^{*} = \sum_{n=0}^{+\infty} (\mu^*)^{\star n}$ denotes  its potential. 

Set $C:=\sup\{y\geq 1: \mu(-y)>0\}$. The support of $\mu^*$ equals $\mathbb Z^-=\mathbb Z\cap (-\infty, 0)$ when $C=+\infty$, otherwise it is $\{-C, \ldots, -1\}$; furthermore, $U^*(-w)>0$ for any $w\geq 0$.  Then, $\mathcal  R(x, y)>0$  if and only if  $y \in \mathbb S_r$, where  $\mathbb S_r=\mathbb N_0\setminus\{0\} $  when $C=+\infty$ and   $\mathbb S_r=\{1, \dots, C\}$ otherwise. Consequently, the set $\mathbb S_r$ is the unique  irreducible and ergodic class of  the Markov chain   $(X(r_k))_{k\geq 0}$ and this chain is aperiodic on  $\mathbb S_r$.

The measure $\nu$ on $\mathbb N_0$   defined by 
\begin{align*}
\nu(x) = \sum_{y=1}^{+\infty} \Big( \frac 12 \mu^*(-x) + \mu^*\big( (-x-y, -x)\big) + \frac 12 \mu^*(-x-y)\Big)\mu^*(-y),
\end{align*}
is, up to a multiplicative constant,  the unique stationary  measure for $(X(r_k))_{k\geq 0}$; its support equals $\mathbb S_r$ (see Theorem 3.6 \cite{PW06}).   

\noindent  Notice that this measure $\nu$ is finite  when $\mathbb E[\xi_n]=0$ and  $\bE\big[|S(\ell_1)|^{1/2}\big]<+\infty$ (and in particular when  $\mathbb E[\xi_n]=0$ and $\mathbb E[\vert \xi_n\vert^{3/2}]<+\infty$ \cite{PW}). {\bf In this case, we normalize $\nu$ it in such a way it is a probability measure.}

\subsection{On the spectrum of the transition probabilities matrix $\mathcal R$}

Let us  recall some spectral properties of the matrix $\mathcal R=(\mathcal R(x, y))_{x, y \in \mathbb N_0}$. By Property 2.3 in  \cite{EP},  the matrix $\mathcal R$ is quasi-compact  on the space $L^{\infty}(\mathbb N_0)$   of bounded functions on $\mathbb N_0$, with  $1$  as the unique  (and simple) dominant eigenvalue; in particular, the rest of the spectrum of $\mathcal R$ is included in a disc with radius $<1$. 

It is of interest in the next section to let $\mathcal R$ act  on a bigger space  than $L^{\infty}(\mathbb N_0)$. 
For instance, following  \cite{EP}, we may  fix $K>1$ and  consider the Banach space 
\[L_K(\mathbb N_0):= \{\phi: \mathbb N_0\to \mathbb C :  \Vert \phi\Vert_K:=\sup_{x \geq 0} \vert \phi(x)\vert \slash  K^x<+\infty\}
\]
endowed with the norm $\Vert \cdot \Vert_K$.   By Property 2.3 in  \cite{EP},  if $\displaystyle \sum_{x \geq 0} K^x \mu(x)<+\infty$ then  $\mathcal R$ acts as a compact operator on  $L_K(\mathbb N_0)$.

In this article, we only assume that $\mu $ has a  finite moment of order $2$ and its negative part has moment of order 3. Consequently,  we consider a smaller Banach space  $\mathcal B_\alpha$ adapted to these hypotheses
and defined by: for $\alpha >0$ fixed, 
\[\mathcal B_\alpha:= \Big\{\phi: \mathbb N_0\to \mathbb C :  \vert \phi\vert_\alpha:=\sup_{x \geq 0} \frac{\vert \phi(x)\vert}{  1+x^\alpha} <+\infty \Big \}.
\]
Endowed with the norm $\vert \cdot \vert_\alpha$, the space   $\mathcal B_\alpha$ is a Banach space on $\mathbb C$.  \begin{proposition}\label{propspectral}
	Fix $\alpha >0$ and 
	assume   $\mathbb E[\xi_n^2]+\mathbb E[ (\xi_n^-)^{2+\alpha}]<+\infty $ and $\mathbb E[\xi_n]=0$.
	Then, the operator $\mathcal R$ acts on  $ \mathcal B_\alpha$ and   $\mathcal R(\mathcal B_\alpha)\subset L^\infty(\mathbb N_0)$. 
	Furthermore, 
	\begin{enumerate}
		\item $\mathcal R$ is compact on $\mathcal B_\alpha$ with spectral radius $1$; 
		\item $1$ is the unique eigenvalue of $\mathcal R$ with modulus $1$, it is simple with corresponding eigenspace $\mathbb C {\bf 1}$;
		\item the rest of the spectrum of  $\mathcal R$ on $\mathcal B_\alpha$ is included in a disc with radius $<1$.
	\end{enumerate}
	
\end{proposition}
Let $\Pi$ be the projection from $\mathcal B_\alpha$ onto the eigenspace $\mathbb C {\bf 1}$ corresponding to this spectral decomposition, i.e. such that $\Pi \mathcal R= \mathcal R \Pi= \Pi$. 
In  other words, there exists a bounded operator $\mathcal Q$  on $\mathcal B_\alpha$  with spectral radius $<1$ such that  $\mathcal R$ may be decomposed as follows:
\begin{equation*}\label{decomposition}
\mathcal R= \Pi + \mathcal Q, \quad  \Pi  \mathcal Q=  \mathcal Q\Pi =0 \quad \text{with} \quad \Pi(\cdot)= \nu(\cdot) {\bf 1}.
\end{equation*}

In the next section, we require that $\mathcal B_\alpha$ does contain the  descending  and ascending renewal functions $h$  and $\tilde h$ of the random walk $S$. This imposes in particular that  $\alpha$ is greater or equal to $1$. 

\begin{proof}
	(1) By (\ref{matrixR}), for any $\phi \in \mathcal B_\alpha$ and $x \geq 0$,
	\[
	\mathcal R\phi(x)=\sum_{y\geq 1}\sum_{w= 0}^x U^*(-w)\mu^*(w-x-y)\phi(y)
	\]
	with 
	$\displaystyle  U^*(-w) = \sum_{n=0}^{+\infty} \pr[S(l_n) = -w] = \pr\Bigl[\cup_{n\geq 0} [S(l_n) = -w]\Bigr]\leq 1.$  
	Therefore, 
	\begin{align*}
	\vert \mathcal R\phi(x)\vert&\leq  
	\sum_{y\geq 1}\sum_{w= 0}^x \mu^*(w-x-y)\vert \phi(y)\vert 
	\\
	&
	\leq \sum_{y\geq 1}\mu^*((-\infty, -y)) \vert \phi(y)\vert 
	\\
	&\leq \left(\sum_{y\geq 1}(1+y^\alpha) \mu^*((-\infty, -y))\right)\vert \phi\vert_\alpha. 
	\end{align*}
	By Theorem 1 in \cite{ChowLai}, the condition $\mathbb E[ (\xi_n^-)^{2+\alpha}]<+\infty$ implies  $\mathbb E \left[|S(\ell_1)|^{1+\alpha}\right]<+\infty$; hence, 
	$$\displaystyle  \sum_{y\geq 1}(1+y^\alpha) \mu^*((-\infty, -y))\leq 
	\bE \left[|S(\ell_1)|\right]+ \bE \left[|S(\ell_1)|^{1+\alpha}\right] <+\infty.$$
	Consequently,  
	\begin{equation}\label{normR}
	\vert \mathcal R\phi\vert_\alpha\leq \vert \mathcal R\phi\vert_\infty\leq \Bigl(\bE \left[|S(\ell_1)|\right]+ \bE \left[|S(\ell_1)|^{1+\alpha}\right] \Bigr) \vert \phi\vert_\alpha
	\end{equation}
	which proves that $\mathcal R$ acts on $\mathcal B_\alpha$ when $\mathbb E[ (\xi_n^-)^{2+\alpha}]<+\infty$.  More precisely, the operator $\mathcal R$ is  bounded from $\mathcal B_\alpha$ into $L^{\infty}(\mathbb N_0)$ and  since the canonical injection $L^{\infty}(\mathbb N_0) \hookrightarrow\mathcal B_\alpha$ is compact, the operator $\mathcal R$ is   compact on $\mathcal B_\alpha$.
	
	Let us now check that  $\mathcal R$ has spectral radius $\rho_\alpha=1$ on $\mathcal B_\alpha$.  On the one hand, the equality $\mathcal R {\bf 1} = {\bf 1}$, with ${\bf 1} \in \mathcal B_\alpha$, yields $\rho_\alpha \geq 1$. On the other hands, $\mathcal R$ is a power bounded operator on $\mathcal B_\alpha$, which readily implies $\rho_\alpha \leq 1$; indeed, 
	for any $n \geq 1$, 
	\[
	\vert \mathcal R^n\phi(x)\vert  \leq \sum_{z =0}^{+\infty} \mathcal R^{n-1}(x, z)\vert \mathcal R \phi(z)\vert
	\leq
	\vert \mathcal R\phi\vert_\infty\sum_{z =0}^{+\infty} \mathcal R^{n-1}(x, z)=\vert \mathcal R\phi\vert_\infty,
	\]
	which yields, combining with (\ref{normR}), 
	\[
	\vert \mathcal R^n\phi\vert_\alpha \leq \vert \mathcal R^n\phi\vert_\infty\le \Bigl(\bE \left[|S(\ell_1)|\right]+ \bE \left[|S(\ell_1)|^{1+\alpha}\right] \Bigr) \vert \phi\vert_\alpha.
	\]
	Consequently, denoting $\Vert\mathcal R^n\Vert_\alpha$   the norm of  $\mathcal R^n$ on $\mathcal B_\alpha$, it holds
	\[
	\sup_{n \geq 0} \Vert \mathcal R^n\Vert_\alpha \leq\Bigl(\bE \left[|S(\ell_1)|\right]+ \bE \left[|S(\ell_1)|^{1+\alpha}\right] \Bigr)  <+\infty.
	\]
	This achieves the proof of assertion 1. 
	
	(2)  Let us control the peripherical spectrum of $\mathcal R$ in $\mathcal B_\alpha$. Let $\theta \in \mathbb R$ and $\phi \in \mathcal B_\alpha$ such that  $\mathcal R\phi=e^{i\theta} \phi.$  
	
	By (\ref{normR}), the function  $\mathcal R\phi$  is bounded, so is  $\phi$. 
	Furthermore, the operator $\mathcal R$ being positive, it holds $\vert \phi\vert \leq \mathcal R \vert \phi\vert$. 
	Consequently, the function  $\vert \phi\vert_\infty-\vert \phi\vert$    is super-harmonic and non-negative, hence constant since the Markov chain $(X(r_n))_{n \geq 0}$ is irreducible and recurrent on this set. 
	
	Without loss of generality, we may assume $\vert \phi\vert =1$ on $\mathbb S_r$, i.e  $\phi (x)= e^{i\varphi(x) }$ for  any $x \in \mathbb S_r$, with $\varphi: \mathbb S_r\to \mathbb R$. Equality  $\mathcal R\phi=e^{i\theta} \phi$ may be rewritten as: for any $x \in \mathbb S_r,$
	\[
	\sum_{y \in \mathbb S_r} e^{i(\varphi(y)-\varphi(x))}\mathcal R(x, y)=e^{i \theta}.
	\]
	Recall that $\mathcal R(x, y)>0$ for any $x, y \in \mathbb S_r$;  thus, by convexity, $e^{i(\varphi(y)-\varphi(x))}=e^{i\theta} $ for any $x, y \in \mathbb S_r$. Thus, $e^{i\theta}=1$ and  the function $\phi$ is harmonic on $\mathbb S_r$,   hence constant.
	Eventually, the function $\phi$ is constant on $\mathbb N_0$: this is the consequence of  
	equality $\mathcal R\phi(x)=e^{i\theta} \phi(x)=\phi(x)$,   valid for any $x \in \mathbb N_0$,  combined with the facts  that $\mathcal R(x, y) >0$ if and only if $y \in\mathbb S_r$ and that $\phi$ is constant on $\mathbb S_r$.

	(3) Assertion 3 is a consequence of assertion 2 and the compactness of $\mathcal R$ on $\mathcal B_\alpha$.
	
\end{proof}
\subsection{ A Renewal limit theorem for the times of reflections}
In this section, we prove the analogous of Corollary \ref{cororenewal} for the process $(r_n)_{n \geq 0}$.
Let us introduce some notations and conventions. 

From  now on,       we focus on the process $(X(n))_{n \geq 0}$ and denote 
$$
((\mathbb N^0)^{\otimes \mathbb N}, 
(\mathcal P(\mathbb N^0))^{\otimes \mathbb N}, (X(n))_{n \geq 0},  (\mathbb P_x)_{x \in \mathbb N_0}, \theta)
$$ the canonical space associated to this process, that is the space of trajectories of the Markov chain $(X(n))_{n \geq 0}$. In particular, $\mathbb P_x, x \in \mathbb N_0$, denotes the {\it conditional probability} with respect to the event $[X(0)=x]$ and $\mathbb E_x$ the corresponding {\it conditional expectation}. The operator $\theta$ is the classical {\it shift transformation} defined by: for any   $(x_k)_{k \geq 0} \in (\mathbb N^0)^{\otimes \mathbb N}$,
\[
\theta ((x_k)_{k \geq 0})=((x_{k+1})_{k \geq 0}.
\] 
For   $n\geq 1$ and $ x, y \geq 0$, set
\begin{equation*}\label{Rnxy}
R_n(x, y):= \mathbb P_x[r_1 =n, X(n)=y],
%=\mathbb P_x[r_1 =n, X(n)=y] 
\end{equation*}
and
\begin{equation*}\label{Sigmanxy}
\Sigma_n(x, y):= \sum_{k=1}^{+\infty}\mathbb P_x[r_k =n, X(n)=y]. 
%=\sum_{l=0}^{+\infty}\mathbb P_x[r_k=n, X(n)=y]
\end{equation*}
We are interested in the behavior as $n \to +\infty$ of these quantities. It   has been already studied  in \cite{MarcLong} (see Lemma 7) for the Lindley process. For the reflected random walk, the argument is   more complicated since the position at time $r_k$ may vary, so that the excursions of the random walk $(X(n))_{n \geq 0}$ between two successive reflection times are not independent. This explain why  we focus here on the reflection process and  it is of interest 
to  express quantities  $R_n(x, y)$ and $\Sigma_n(x, y)$ in terms of operators and product of operators related to this sub-process.

We consider  the linear operators  $R_n: L^{\infty}(\mathbb N_0) \to L^{\infty}(\mathbb N_0), n \geq 0$,  defined by: for any $\phi \in L^{\infty}(\mathbb N_0)$ and $x \geq 0$, 
\[
R_n \phi(x) = \sum_{y \geq 1} R_n(x,y)\phi(y)=\mathbb E_x[  r_1=n; \phi(X( n))].
\]
In particular,  $R_n(x, y)=R_n \1_{\{y\}}(x)$. 
The quantity $\Sigma_n(x, y)$ is also expressed in terms of the $R_k$ as follows:  
\begin{align}\label{renewalsigma}
\Sigma_n(x, y)&= \sum_{k=1}^{+\infty}\mathbb P_x[r_k=n, X(n)=y] \notag
\\
&= \sum_{k=1}^{+\infty} \sum_{j_1+\cdots + j_k = n} \pr_x[r_1 =j_1, r_2  - r_1 = j_2, \ldots, r_{k}  - r_{k-1}  = j_k, X(n) = y] \notag\\
&= \sum_{k=1}^{+\infty} \sum_{j_1+\cdots + j_k = n} 
R_{j_1}\ldots R_{j_k}\1_{\{y\}}(x)
\end{align}
Firstly, let us check that 
the   $R_n$ act  on $\mathcal B_\alpha$.   

\begin{lemma} \label{normeBalphaRn}There exists a  positive constant  $C_4$ such that, for any $n \geq 1$ and $\alpha >0$,  
	\begin{equation*}
	\label{Rn_C3}
	\vert R_n\vert_ \alpha  \leq C_4 \frac{ \bE\left[ (\xi_n^-) ^{2+\alpha} \right]}{n^{3/2}}.
	\end{equation*} 
\end{lemma}
\begin{proof} 
	For any $\phi \in \mB_\alpha$ and $x \geq 0$, 
	\begin{align*}
	|R_n\phi(x)|  &\leq \sum_{y \geq 1} |\phi(y)| \pr_x[r_1 = n, X(n)= y] \\
	& = \sum_{y \geq 1} \sum_{z \geq 0} |\phi(y)|  \pr[\tau^S(x) \geq n-1, x + S(n-1) = z, z + \xi_n = -y] \\
	& = \sum_{y \geq 1} \sum_{z \geq 0}  |\phi(y)|  \pr[\tau^S(x) \geq n-1, x + S(n-1) = z] \pr[\xi_n = -y-z].
	\end{align*}	
	Hence, by Lemma \ref{LemA},
	\begin{align*}
	\frac{|R_n\phi(x)|}{1+x^\alpha} &\preceq \frac{1}{n^{3/2}}\sum_{y \geq 1} \sum_{z \geq 0} |\phi(y)| \frac{h(x)}{1+ x^\alpha}  \tilde{h}(z) \pr[\xi_1 = - y - z].
	\end{align*}
	Since $h(x)= O(x)$ and $\tilde{h}(z) = O(z)$, 
	\begin{align*}
	\frac{|R_n\phi(x)|}{1+x^\alpha}  &\preceq  \frac{\vert\phi\vert_ \alpha}{n^{{3/2}}}  \sum_{y \geq 1} \sum_{z \geq 0} (1+y^\alpha )\tilde{h}(z) \pr[\xi_1 = - y - z]\\
	&\preceq \frac{\vert\phi\vert_ \alpha}{n^{{3/2}}}  \sum_{y \geq 1} \sum_{z \geq 0} (1+y^\alpha )z \pr[\xi_1 = - y - z]\\
	&= \frac{\vert\phi\vert_ \alpha}{n^{{3/2}}}  \sum_{t \geq 1} \sum_{y = 1}^{t} (1+y^\alpha )(t-y) \pr[\xi_1 = - t]\\
	&\preceq  \frac{\vert\phi\vert_ \alpha}{n^{{3/2}}}  \sum_{t \geq 1} t^{2+\alpha}\pr[\xi_1 = - t],
	%	  \\
	%	& = \bE[ \max(0, -\xi_1)^{2+\alpha}]\frac{\vert\phi\vert_ \alpha}{n^{{3/2}}}.	
	\end{align*}
	%	Therefore, 
	%	$$\vert R_n\phi\vert_ \alpha   \preceq  \frac{ \bE[ \max(0, -\xi_1)^{2+\alpha}  ]}{n^{{3/2}}}\vert\phi\vert_ \alpha,$$
	which achieves the proof. 
\end{proof}
Hence, $\displaystyle \sum_{n \geq 1} \vert R_n\vert_\alpha<+\infty$; in particular, the sequence $ (\sum_{n=1}^N R_n)_{N\geq 1}$ converges  in $\mathcal B_\alpha$. Note that its limit equals  $\mathcal R$ in $\mathcal B_\alpha$;  
indeed, 
	$$\sum_{n\geq 1} R_n \phi(x) = \sum_{n \geq 1} \mathbb{E}_x[\phi(X(n)), r_1 = n] = \mathbb{E}_x[\phi(X(r_1))]= \mathcal{R}\phi(x).$$
	We can write 
$\mathcal R= \sum_{n\geq 1} R_n$  
and, for any $z\in \overline{\mathbb D}:= \{z \in \mathbb C: \vert z\vert \leq1\}$, we set 
$$\mathcal R(z)= \sum_{n \geq 1} z^nR_n.$$

\begin{proposition} \label{renewalsequence}
	Fix $\alpha >0$ and assume $\mathbb E[ \xi_n^2]+\mathbb E[(\xi_n^-)^{2+\alpha}]<+\infty$ and $\mathbb E[\xi_n]=0$.
	The sequence $(R_n)_{n \geq 0}$  is an   {\bf aperiodic renewal  sequence of operators}, i.e. it satisfies the  following properties   (see \cite{Gouezel}):

	{\rm  (R1). }  The operator $\mathcal R=\mathcal R(1)$ has a simple eigenvalue at $1$ and the rest of its spectrum is contained in a disk of radius $<1$. 
	
	{\rm  (R2).}    For any $n \geq 1$, set ${\bf r}_n:= \nu R_n\1=\sum_{x\geq 1} \nu(x) \mathbb P_x(r_1=n);$ hence,   
	$$
	\Pi R_n \Pi = {\bf r}_n \Pi,
	$$
	where  $\Pi$ denotes  the eigenprojection of $\ \mathcal R $ for the eigenvalue 1.

	{\rm  (R3).}  There exists a constant  ${\bf C}>0$ such that  $\quad  \vert R_n \vert_\alpha \leq {{\bf C}\over  n^{{3/2}}}.$ 
	
	{\rm  (R4).}  $  \sum_{j>n} {\bf r}_j \sim {{\bf c}\over \sqrt{n} }$ with ${\bf c}= c_1\nu(h)$, where $c_1$ is  the positive constant given by Lemma \ref{LemA} and $h$ is the descending renewal function   of the random walk $S$.
	
	{\rm (R5).}  The spectral radius of $\ \mathcal R(z)$ is strictly less than $1$ for $z \in \overline{\mathbb D}\setminus\{1\}$. 
	
\end{proposition}
\begin{proof}
	(R1) is a direct consequence of  Proposition \ref{propspectral}.
	
	\noindent (R2) Recall that $\Pi \phi =  \nu(\phi)   \1$ for any $\phi \in \mathcal B_\alpha$. Hence, setting $g_n(x):= \mathbb P_x (r_1=n)$, it holds
	$R_n \Pi \phi   = \nu(\phi) g_n$, thus 
	$$\Pi R_n \Pi \phi   = \nu(\phi) \Pi(g_n)=\sum_{x\geq 1} \nu(x) \mathbb P_x(r_1=n) \nu(\phi)\1,$$
	which is the expected result.
	
	\noindent (R3) follows from Lemma \ref{normeBalphaRn}.
	
	\noindent (R4)  Thanks to  Lemma \ref{LemA},    
	\begin{align*}
	\sum_{j\geq n} \mathbf{r}_j &=  \sum_{x \geq 1} \sum_{j\geq n} \nu(x)\pr_x[r_1 = j] =  \sum_{x \geq 1} \nu(x)\pr_x[r_1  \geq n]  
	\sim  c_1 \frac{ \nu(h)}{\sqrt{n}} \quad \text{as} \quad n \to +\infty.
	\end{align*}
	Notice that  $0<\nu(h)<+\infty$  since $\bE[|S(\ell_1)|] <+ \infty$; indeed, $1\leq h(x)= O(x)$ and 
	\begin{align*}
	\sum_{x \geq 1}x\nu(x) &   \preceq  \sum_{x\geq 1} \sum_{y \geq 1} \sum_{w=x}^{x+y} \mu^*(-w) \mu^*(-y)x 
	=  \sum_{y \geq 1} \sum_{w\geq 1} \mu^*(-w) \mu^*(-y) \sum_{x= (w-y)\vee 0}^{w} x \\
	&\leq \sum_{y \geq 1} \sum_{w\geq 1} yw \mu^*(-w) \mu^*(-y) = \left(\sum_{y\geq 1} y\mu^*(-y) \right)^2 =  \left(\bE[|S(\ell_1)|] \right)^2<+\infty.
	\end{align*}
	
	\noindent (R5)	The argument is the same as the one used to control the peripherical spectrum of $\mathcal R$ in Proposition  \ref{propspectral}. For any $z \in \overline{\mathbb D}\setminus\{1\}$,  the operators $\mathcal R(z)$ are compact on $\mathcal B_\alpha$, with spectral radius $\rho_z\leq 1$.  
	
	If $\rho_z= 1$, there exist $\theta \in \mathbb R$ and $\phi \in \mathcal B_\alpha$ such that  $\mathcal R(z)\phi=e^{i\theta} \phi.$  
	Hence $\vert \phi\vert= \vert \mathcal R(z) \phi\vert  \leq \mathcal R\vert \phi\vert$ and since $\mathcal R(\mathcal B_\alpha)\subset L^\infty(\mathbb N_0)$, the function $\vert \phi\vert $ is bounded on $\mathbb N_0$, thus  constant on $\mathbb S_r$.

	Without loss of generality, we may assume $\vert \phi\vert =1$ on $\mathbb S_r$, i.e  $\phi (x)= e^{i\varphi(x) }$ for  any $x \in \mathbb S_r$, with $\varphi: \mathbb S_r\to \mathbb R$. Equality  $\mathcal R(z)\phi=e^{i\theta} \phi$ may be rewritten as: for any $x \in \mathbb S_r,$
	\[
	\sum_{n \geq 1} \sum_{y \in \mathbb S_r} z^n e^{i\varphi(y)}\mathbb P_x(r_1=n; X(n)=y)=e^{i \theta}e^{i\varphi(x)}.
	\]
	By convexity, since $\sum_{n \geq 1} \sum_{y \in \mathbb S_r}  \mathbb P_x(r_1=n; X(n)=y)=1$,  we obtain: for all $n \geq 1$ and $x, y \in \mathbb S_r$,
	\[
	z^ne^{i\varphi(y)}=e^{i \theta}e^{i\varphi(x)}.
	\]
	Setting $x=y$, it yields $z^n= e^{i \theta}$, so that $z^n$ does not depend on $n$. Finally $z=1$. Thus, $\rho_z<1$ when $z \in \overline{\mathbb D}\setminus\{1\}$.

\end{proof}

By  (R5), for $\vert z\vert  <1$, the operator $T(z):= (I-\mathcal R(z))^{-1}$ is well defined in $\mathcal B_\alpha$; a direct formal computation yields 
$T(z)= \sum_{n=0}^{+\infty} T_n z^n$,   where the
$T_n$ are bounded operators on $\mathcal B_\alpha$   defined   by:  
$$T_0=I \qquad \text{and}\qquad T_n = \sum_{k=1}^{+\infty} \sum_{j_1+\cdots + j_k = n} R_{j_1}\cdots R_{j_k}\quad \text{for} \quad n \geq 1.$$

The so-called {\it renewal equation}    $T(z):= (I-R(z))^{-1}$  is of fundamental importance to understand the asymptotics of the $T_n$, several functional analytic tools can be brought into play.
Such sequences of operators  $(R_n)_{n \geq 0}$ and $(T_n)_{n \geq 0}$  have been  the object of many studies, related to renewal theory in a non-commutative  setting.  We refer to the paper \cite{Gouezel}, which fits perfectly  here. 
The following statement is  analogous of the last assertion of Corollary \ref{cororenewal} for  the reflected random walk. 
\begin{corollary} \label{G4.1}
	The sequence $(\sqrt{n}  T_n)_{n \geq 1}$ converges in $\mathcal B_\alpha$ towards the operator ${1\over \pi c_1 \nu(h)} \Pi$.
\end{corollary}
\begin{proof} Apply Theorem 1.4 in \cite{Gouezel} with $\beta=1/2$ and $\ell(n)= {\bf c} = c_1\nu(h)$.
\end{proof}

As a direct consequence, by equality (\ref{renewalsigma}), it holds 
\[
\lim_{n \to +\infty} \sqrt{n}\Sigma_n(x,y) ={\nu(y)\over \pi c_1 \nu(h)}. 
\]

In the next section,  we have to consider and study some  modifications of the $\Sigma_n(x,y)$ which we introduce now. 
For  any $x \geq 0$ and $0< s< t<1$,  
\begin{align*}
\widehat{\Sigma}_n(x,t,s) &:= n \sum_{l\geq 0} \pr_x[r_l = [ns], r_{l+1} > [nt]],
%\\
%&= n \sum_{y \geq 0} \sum_{l\geq 0} \pr_x[r_l = [ns], X([ns]) = y] \pr[\tau^S(y) > [nt]-[ns]]
\end{align*}
and
\begin{align*}
\widetilde{\Sigma}_n(x,t,s)& :=    n^2 \sum_{l=0}^{+\infty} \pr_x \left[ r_l= [ns], r_{l+1} = [nt] \right].
\end{align*}
These quantities appear in a natural way to control the finite distribution of the process $(X_n(t))_{n \geq 0}.$

\section{Proof of Theorem \ref{meander}}
From now on,  we fix  $\alpha = 1; $  this implies that $h\in \mathcal B_\alpha$, which is necessary from now on (see Lemmas \ref{lem:M} and \ref{lem:M'}).

\subsection{One-dimensional distribution} 
We fix a  bounded and Lispchitz continuous function $\phi: \mathbb R\to \mathbb R$.

\begin{lemma}\label{1-dim}
	For any $t\in [0, 1]$ and $x \geq 0$, it holds
	$$\lim_{n\to +\infty}\bE_x \left[ \phi \left(X_n(t)\right)\right] 
	=\int_0^{+\infty} \phi(u) \frac{2e^{-u^2/2t}}{\sqrt{2\pi t}}du = \bE[\phi(|B_t|)],$$
	where $B$ is a standard Brownian motion. 
\end{lemma}
\begin{proof}	
	We fix  $t\in (0, 1)$ and    decompose the expectation $\displaystyle \bE \left[ \phi \left( \frac{X({[nt]})}{{\sigma}\sqrt{n}}\right)\right] $ as follows: 
	\begin{align*}
	&\bE_x \left[ \phi \left( \frac{X({[nt]})}{{\sigma}\sqrt{n}}\right)\right] \\
	&\approx  \sum_{k=0}^{[nt]-1} \sum_{l\geq 0}  \bE_x \Bigl[ \phi \left( \frac{X({[nt]})}{{\sigma}\sqrt{n}}\right);
	\\
	&\hskip 3cm 
	 r_l = k, X(k) + \xi_{k+1} \geq  0, \ldots, X(k) + \xi_{k+1} + \cdots + \xi_{[nt]} \geq 0 \Bigr]\\
	&=  \sum_{k=0}^{[nt]-1}  \sum_{y\geq 0}  \Sigma_k(x, y) \bE \Bigl[ \phi \left( \frac{y+\xi_{k+1} + \ldots + \xi_{[nt]} }{{\sigma}\sqrt{n}}\right);
	\\
	&\hskip 3cm 
	 y + \xi_{k+1} \geq  0, \ldots, y + \xi_{k+1} + \cdots + \xi_{[nt]} \geq 0 \Bigr]\\
	&=  \sum_{k=0}^{[nt]-1}  \sum_{y\geq 0} \Sigma_k(x, y)   \bE \left[ \phi \left( \frac{y+S([nt]-k) } {{\sigma}\sqrt{n}}\right)| \tau^S(y) > [nt]-k \right] \\
	&\hskip 5cm \times  \pr\left[\tau^S(y) > [nt] - k\right].		
	\end{align*}

	For each $k = 2, \ldots, [nt]-4 $ and  any $s \in [\frac kn, \frac{k+1}{n})$, 
	\begin{align*}
	f_n(s) &=  n\sum_{y\geq 0} \Sigma_{[ns]}(x, y) \bE \left[ \phi \left( \frac{y+S([nt]-[ns]) } {{\sigma}\sqrt{n}}\right)| \tau^S(y) > [nt]-[ns] \right] \\
	 \\
	&\hskip 3cm  \times  \pr\left[\tau^S(y) > [nt] - [ns] \right],
	\end{align*}
	and $f_n(s)= 0$ on $[0, {2\over n})$ and $[{[nt]-1\over n}, t)$. Hence,
	\begin{align*}
	\bE_x \left[ \phi \left( \frac{X({[nt]})}{\sigma\sqrt{n}}\right)\right] & = \int_0^{t} f_n(s)ds +{\mathcal O}\left({1\over \sqrt{n}}\right).
	\end{align*}
	Now, let us
	set : for $n\geq 1$ and  any $y \in \mathbb N_0$, 
	\begin{align*}
	a_n(y) &= \Sigma_{[ns]}(x, y)   \pr\left[\tau^S(y) > [nt] - [ns] \right],
	\\
	b_n(y) &=\bE \left[ \phi \left( \frac{y+S([nt]-[ns]) } {{\sigma}\sqrt{n}}\right)| \tau^S(y) > [nt]-[ns] \right].
	\end{align*}
	For any  $n\geq 1$,  it holds
	\[
	\sum_{y\geq 0} a_n(y)   = n \sum_{l\geq 0} \pr_x[r_l = [ns], r_{l+1} > [nt]]=: \widehat{\Sigma}_n(x,t,s),\\
	\]
	and $|b_n(y)| \leq \vert\phi\vert_\infty$. The two  following lemmas allow us to control the behavior as $n \to +\infty$ of the integral  $\displaystyle \int_0^{t} f_n(s)ds$; the proof of Lemma \ref{lem:M}   is postponed to the last section, the one of \ref{lem:prod} is straightforward. 
	\begin{lemma} \label{lem:M} 
		For each $0< s< t<1$,   
		\[
		\lim_{n\to+ \infty} \widehat{\Sigma}_n(x,t,s) = \frac{1}{\pi \sqrt{s(t-s)}}.
		\]
		Moreover, there exists a positive constant $C_5$ such that
		\[
		\widehat{\Sigma}_n(x,t,s) \leq  C_5 \frac{ 1+ x }{\sqrt{s(t-s)}} \ \text{ for all } \  0<s<t<1 \ \text{ and } \ x \in \n.
		\]
		%b) For each $0<s<t< 1$, it holds that 
		%\begin{align} \label{M3}
		%\lim_{n\to+ \infty} \widetilde{\Sigma}_n(x,t,s) = \frac{2}{\pi\sqrt{s(t-s)^3}}.
		%\end{align}
		%Moreover, there exists a positive constant $C_7$ such that 
		%\begin{align} \label{M4}
		%\widetilde{\Sigma}_n(x,t,s) \leq  \frac{2}{\pi\sqrt{s(t-s)^3}} + \frac{C_7(1+x)}{\pi\sqrt{s(t-s)^3}}  \ \text{ for all } \  0<s<t<1, \ \text{ and } \ x \in \n.
		%\end{align}
	\end{lemma} 
	\begin{lemma} \label{lem:prod} 
		Let  $(a_n(y))_{y \in \n_0^k}, (b_n(y))_{y \in \n_0^k}$ be arrays of real numbers for some integer $k\geq 1$. Suppose that 
		\begin{itemize}
			\item $a_n(y) \geq 0$;
			\item $\displaystyle \lim_{n\to+ \infty} \sum_{y \in \n_0^k } a_n(y) = A$;
			%\item $\sup_n \sum_{y\geq 0} a_n(y) = A_1 <+ \infty$;
			\item $\displaystyle\lim_{n\to+ \infty} b_n(y)= B$ for all $y \in \n_0^k$;
			\item $\displaystyle\sup_{ n\geq 1, y \in \n_0^k} |b_n(y)|   <+ \infty$. 
		\end{itemize}
		Then $$\lim_{n\to+ \infty} \sum_{y\geq 0} a_n(y)b_n(y) = AB.$$
	\end{lemma}
	Lemmas \ref{LemC2},  \ref{lem:M} and  \ref{lem:prod} combined altogether yield: for any $s\in (0,t)$,  
	\begin{align*}
	\lim_{n\to +\infty} f_n(s) &=   {1\over  {\pi}} \frac{1}{\sqrt{s(t-s)}}\int_0^{+\infty} \phi(z\sqrt{t-s})ze^{-z^2/2}dz.	 
	\end{align*}
	Moreover, 
	$$\sup_n |f_n(s)| \leq  \displaystyle C_5\frac{ 1+ x}{\sqrt{s(t-s)}} \vert\phi\vert_\infty   =: \hat{f}(s).$$
	Since  $\hat{f} \in L^1[0,t]$, the Lebesgue dominated convergence theorem yields 
	\begin{align*}
	\lim_{n\to +\infty}\bE \left[ \phi \left( \frac{X({[nt]})}{\sigma\sqrt{n}}\right)\right] &= \lim_{n\to +\infty} \int_0^{t}f_n(s) ds \\
	&= {1\over  {\pi}}\ \int_0^t \frac{1}{\sqrt{s(t-s)}} \left(\int_0^{+\infty} \phi(z\sqrt{t-s})ze^{-z^2/2}dz \right)ds \\
	&=\int_0^{+\infty} \phi(u) \frac{2e^{-u^2/2t}}{\sqrt{2\pi t}}du, 
	\end{align*}
	where the last equation follows from the identity (\cite{ItoMcKean}, p. 17)
	\begin{equation}\label{formulaIMK}
	\int_0^{+\infty} \frac{1}{\sqrt{ t}} \exp \Big( -\alpha t - \frac {\beta}{t}\Big)dt = \sqrt{\frac{\pi}{\alpha}}e^{-2\sqrt{\alpha \beta}} \quad (\alpha, \beta >0)
	\end{equation}
	and some change of variable computation. 
	We achieve the proof of  Lemma \ref{1-dim} by noting  that, since $\phi$ is Lipschitz continuous (with Lipschitz coefficient $[\phi]$), 
	\begin{align} \label{lagtime} 
	\left| 	\bE_x \left[ \phi \left( \frac{X({[nt]})}{\sigma\sqrt{n}}\right)\right] - \bE_x \left[ \phi \left(X_n(t)\right)\right] \right| 
	&\leq [\phi] \bE_x \left[ \left| 	 \frac{X({[nt]})}{\sigma\sqrt{n}}- X_n(t) \right|\right] \notag \\
	& \leq \frac{1}{\sigma\sqrt{n}} [\phi] \bE \left[  	 |\xi_{[nt]+1}| \right] \to 0 \text{ as } n \to+ \infty.  \qquad
	\end{align} 
\end{proof}

\subsection{Two-dimensional distributions}
The convergence of the finite-dimensional distributions of $(X_n(t))_{n \geq 1}$ is more delicate. We detail the argument for two-dimensional ones, the general case may be treated in a similar way. 

Let us   fix $0<s<t, n \geq 1$ and  denote 
$$\kappa =\kappa(n, s)= \min\{k>[ns]: X(k-1)+\xi_k < 0\}.$$
We decompose $\displaystyle  \bE_x \left[ \phi_1\left( \frac{X({[ns]})}{\sigma\sqrt{n}}\right) \phi_2 \left( \frac{X([nt])}{\sigma\sqrt{n}}\right)\right] $ as
\begin{align*}
%&\bE_x \left[ \phi_1\left( \frac{X({[ns]})}{\sigma\sqrt{n}}\right) \phi_2 \left( \frac{X([nt])}{\sigma\sqrt{n}}\right)\right] \\ 
&  \underbrace{\sum_{k = [ns]+1}^{[nt]} \bE_x \left[ \phi_1\left( \frac{X({[ns]})}{\sigma\sqrt{n}}\right) \phi_2 \left( \frac{X([nt])}{\sigma\sqrt{n}}\right)\1_{\{\kappa = k\}} \right]}_{A_1(n)} \\
&\hskip 5cm  + \underbrace{\bE_x \left[ \phi_1\left( \frac{X({[ns]})}{\sigma\sqrt{n}}\right) \phi_2 \left( \frac{X([nt])}{\sigma\sqrt{n}}\right)\1_{\{\kappa > [nt]\}} \right]}_{A_2(n)}.
\end{align*}
The term $A_1(n)$ deals with the trajectories of  the process $X$  which reflect between $[ns]+1$ and $[nt]$ while $A_2 (n)$   concerns the others trajectories.

\subsubsection{Estimate of $A_1(n)$} \label{Sec:3.2.1}
As in the previous section, we decompose $A_1(n)$ as   
\begin{align*}
A_1(n) &= \sum_{k_1 = 0}^{[ns]-1} \sum_{k_2 = [ns]}^{[nt]} \sum_{l=0}^{+\infty} \sum_{y\geq 1} \sum_{z \geq 1} \sum_{w\geq 0}\\
&\hskip 1cm  \bE_x \Bigg[ \phi_1\left( \frac{X({[ns]})}{\sigma\sqrt{n}}\right) \phi_2 \left( \frac{X([nt])}{\sigma\sqrt{n}}\right); \\
&\hskip 1cm r_l = k_1, X(k_1)= z,   z + \xi_{k_1+1} \geq  0, \ldots, z + \xi_{k_1+1} + \cdots + \xi_{k_2-2} \geq 0, \\
& \hskip 1cm  z + \xi_{k_1+1} + \cdots + \xi_{k_2-1} = w, w+ \xi_{k_2} = -y \Bigg] \\
&= \sum_{k_1 = 0}^{[ns]-1} \sum_{k_2 = [ns]}^{[nt]} \sum_{l=0}^{+\infty} \sum_{y\geq 1} \sum_{z \geq 1} \sum_{w\geq 0} \bE_x \Bigg[ \phi_1\left( \frac{z + \xi_{k_1+1} + \cdots + \xi_{[ns]}} {\sigma\sqrt{n}}\right) \\
&\hskip 1cm \times \phi_2 \left( \frac{y+\xi_{k_2+1}+ \cdots + \xi_{[nt]}}{\sigma\sqrt{n}}\right); \\ 
&\hskip 1cm  r_l = k_1, X(k_1)= z,   z + \xi_{k_1+1} \geq  0, \ldots, z + \xi_{k_1+1} + \cdots + \xi_{k_2-2} \geq 0, \\
&\hskip 1cm  z + \xi_{k_1+1} + \cdots + \xi_{k_2-1} = w, w+ \xi_{k_2} = -y \Bigg] \\
&= \sum_{k_1 = 0}^{[ns]-1} \sum_{k_2 = [ns]}^{[nt]} \sum_{l=0}^{+\infty} \sum_{y\geq 1} \sum_{z \geq 1} \sum_{w\geq 0} \\
&\hskip 1 cm\bE_x \Bigg[  \phi_2 \left( \frac{y+\xi_{k_2+1}+ \cdots + \xi_{[nt]}}{\sigma\sqrt{n}}\right)\Bigg]  \pr_x[r_l = k_1, X(k_1)= z]\\   
&\hskip 1 cm \times \bE_x \Bigg[  \phi_1\left( \frac{z + \xi_{k_1+1} + \cdots + \xi_{[ns]}} {\sigma\sqrt{n}}\right),
 \\
&\hskip 3cm   z + \xi_{k_1+1} \geq  0, \ldots, z + \xi_{k_1+1} + \cdots + \xi_{k_2-2} \geq 0,\\
&\hskip 5cm z + \xi_{k_1+1} + \cdots + \xi_{k_2-1} = w, w+ \xi_{k_2} = -y \Bigg].
\end{align*}
Using the  fact that the $\xi_k$ are i. i. d., we obtain 
\begin{align*}
A_1(n) &= \sum_{k_1 = 0}^{[ns]-1} \sum_{z \geq 1}  \Sigma_{k_1}(x, z)  
\sum_{k_2 = [ns]}^{[nt]}  \sum_{y\geq 1}  \sum_{w\geq 0} \bE_y \Bigg[  \phi_2 \left( \frac{X([nt]-k_2)}{\sigma\sqrt{n}}\right)\Bigg] \\   
& \times \bE \Bigg[  \phi_1\left( \frac{z + S([ns] -k_1)} {\sigma\sqrt{n}}\right)| \tau^S(z) > k_2 - k_1 -1, z+ S(k_2 - k_1-1) = w\Bigg]\\
& \times  \pr[\tau^S(z) > k_2 - k_1 -1, z+ S(k_2 - k_1-1) = w] \pr[\xi_{1} = -w-y]. 
\end{align*}
For  any $2\leq k_1< [ns]-6$ and $[ns] \leq   k_2\leq [nt]$ and any $s_1 \in  [ {k_1\over n},  {k_1+1\over n})$ and $s_2 \in  [ {k_2\over n},  {k_2+1\over n})$, we write 
\begin{align*}
f_n&(s_1, s_2) = n^2 \sum_{z \geq 1}  \Sigma_{[ns_1]}(x, z)  
\sum_{y\geq 1}  \sum_{w\geq 0} \bE_y \Bigg[  \phi_2 \left( \frac{X([nt]-[ns_2])}{\sigma\sqrt{n}}\right)\Bigg] \\   
& \times \bE \Bigg[  \phi_1\left( \frac{z + S([ns] - [ns_1])} {\sigma\sqrt{n}}\right)| \tau^S(z) > [ns_2] - [ns_1] -1, \\
&\hskip 5cm z+ S([ns_2] - [ns_1]-1) = w\Bigg]\\
& \times  \pr\left[\tau^S(z) > [ns_2] - [ns_1] -1, z+ S([ns_2] - [ns_1]-1) = w \right] \pr[\xi_{1} = -w-y],
\end{align*}
and $f_n(s_1, s_2)= 0$ for the others values of $k_1 $, such that $0\leq k_1\leq [ns]$. Hence,
\begin{align*}
A_1(n)= \int_0^s ds_1\int _s^t ds_2\  f_n(s_1, s_2)  +{\mathcal O}\left({1\over \sqrt{n}}\right).
\end{align*}
%By  Corollary 2.5 in \cite{CC}, the random walk bridge   conditioned to stay non-negative, starting at $z$ and ending at $w$, under linear interpolation and diffusive rescaling, converges in distribution on $C([0,1],\mathbb R) $ toward the normalized Brownian excursion; in other words, for each $z, w \geq 0$,  
It follows from Lemma \ref{lemCC} that, for each $z,w \geq 0$,
\begin{align*}
&\lim_{n\to   +\infty} \bE \Bigl[  \phi_1\left( \frac{z + S([ns] - [ns_1])} {\sigma\sqrt{n}}\right)| \tau^S(z) > [ns_2] - [ns_1] -1, 
\\
&\hskip 4cm  z+ S([ns_2] - [ns_1]-1) = w\Bigr]\\
&= \int_0^{+\infty} 2\phi_1(u\sqrt{s_2-s_1}) \exp\left(-\frac{u^2}{2\frac{s-s_1}{s_2-s_1}\frac{s_2-s}{s_2-s_1}}\right) \frac{u^2 }{\sqrt{2\pi \frac{(s-s_1)^3}{(s_2-s_1)^3}\frac{(s_2-s)^3}{(s_2-s_1)^3}}}du \\
&= {2\over {\sqrt{2\pi}} }\int_0^{+\infty}  \phi_1(v) \exp\left(-\frac{v^2}{2\frac{(s-s_1)(s_2-s)}{s_2-s_1}}\right) \frac{v^2  }{ \sqrt{{(s-s_1)^3(s_2-s)^3}\over (s_2-s_1)^3}}dv.
\end{align*}
By Lemma \ref{1-dim},   $$\lim_{n\to+ \infty}  \bE_y \left[\phi_2 \left( \frac{X([nt]-[ns_2])}{\sigma\sqrt{n}}\right)\right]=  \int_0^{+\infty} \phi_2(u) \frac{2e^{-u^2/2(t-s_2)}}{\sqrt{2\pi (t-s_2)}}du.$$ 
We set 
\begin{align*}
a_n(x, y,z,w)  &= n^2  \Sigma_{[ns_1]}(x, z)  \\
& \hskip 0.5cm\times  
\pr\left[\tau^S(z) > [ns_2] - [ns_1] -1, z+ S([ns_2] - [ns_1]-1) = w \right] 
\\
&\hskip 0.5cm \times \pr[\xi_{1} = -w-y],\\
b_n(y,z,w) &= \bE_y \Bigg[  \phi_2 \left( \frac{X([nt]-[ns_2])}{\sigma\sqrt{n}}\right)\Bigg]   \\
&\times  
\bE \Bigg[  \phi_1\left( \frac{z + S([ns] - [ns_1])} {\sigma\sqrt{n}}\right)| \tau^S(z) > [ns_2] - [ns_1] -1, \\
&\hskip3 cm z+ S([ns_2] - [ns_1]-1) = w\Bigg]
\end{align*} 
Note that $\sum_{z\geq 1}\sum_{y\geq 1} \sum_{w\geq 0}  a_n(x, y,z,w)  = \widetilde{\Sigma}_n(x,s_2,s_1)$.
The behavior as $n \to +\infty$ of the quantity $\widetilde{\Sigma}_n(x,s_2,s_1)$ is given by the 
following Lemma, whose proof is postponed to the last section.

\begin{lemma} \label{lem:M'} 
	%a) For each $0< s< t<1$, it holds that 
	%\begin{align} \label{M1}
	%\lim_{n\to+ \infty} \widehat{\Sigma}_n(x,t,s) = \frac{1}{\pi \sqrt{s(t-s)}}.
	%\end{align} 
	%Moreover, there exists a positive constant $C_6$ such that
	%\begin{align} \label{M2} 
	%\widehat{\Sigma}_n(x,t,s) \leq \frac{1}{\pi\sqrt{s(t-s)}} + \frac{C_6(1+ x)}{\sqrt{s(t-s)}} \ \text{ for all } \  0<s<t<1, \ \text{ and } \ x \in \n.
	%\end{align} 
	%b)
	For all $0<s<t< 1$, it holds  
	\begin{equation*} \label{M3}
	\lim_{n\to+ \infty} \widetilde{\Sigma}_n(x,t,s) =  \frac{1}{2\pi\sqrt{s(t-s)^3}}.
	\end{equation*}
	Moreover, there exists a positive constant $C_6$ such that, for all  $0<s<t<1$ and $n \geq 0$, 
	\begin{equation*} \label{M4}
	\widetilde{\Sigma}_n(x,t,s) \leq  C_6\frac{ 1+x}{\pi\sqrt{s(t-s)^3}}.
	\end{equation*}
\end{lemma} 

By Lemmas \ref{lem:M'} and \ref{lem:prod}, we get 
$\lim_{n \to+ \infty} f_n(s_1,s_2) = f(s_1,s_2)$ where 
\begin{align*}
f(s_1,s_2) =&{1\over \pi^{2} \sqrt{ s_1}}\int_0^{+\infty}  \phi_1(v)  \exp\left(-\frac{v^2}{2\frac{(s_2-s)(s-s_1)}{s_2-s_1}}\right) \frac{v^2  }{ \sqrt{{(s-s_1)^3(s_2-s)^3} }}dv\\
& \quad \times \int_0^{+\infty} \phi_2(u) \frac{ e^{-u^2/2(t-s_2)}}{\sqrt{  t-s_2}}du.
\end{align*}
Moreover,  following the  argument in the proof of Lemma \ref{1-dim}, we can show that the sequence  $(|f_n|)_{n\geq 1}$ is uniformly bounded  by a function which is  integrable with respect to Lebesgue measure on $[0,s]\times[s,t]$. 
Hence, using again the Lebesgue dominated convergence theorem, we get 
\begin{align*}
& \lim_{n\to +\infty} A_1(n)  
=   \int_{0 }^s ds_1\int_{ s}^t ds_2 f(s_1, s_2)
\\
&= {1\over \pi^{2} }\int_{0 }^s {ds_1\over  \sqrt{ s_1}}\int_{ s}^t ds_2
\int_0^{+\infty}  \phi_1(v) \exp\left(-\frac{v^2}{2\frac{(s_2-s)(s-s_1)}{s_2-s_1}}\right)  \frac{v^2  }{ \sqrt{{(s-s_1)^3(s_2-s)^3}}} 
\\
& \qquad \qquad\qquad \qquad\qquad \qquad\qquad \qquad \times \int_0^{+\infty} \phi_2(u) \frac{e^{-u^2/2(t-s_2)}}{\sqrt{  t-s_2}}dudv, 
\end{align*}
which yields, using again (\ref{formulaIMK}), 
\begin{align} \label{lim:A1} 
&\lim_{n \to+ \infty} A_1(n) = \frac{2}{\pi \sqrt{ s (t-s)}}  \int_0^{+\infty} \int_0^{+\infty} \phi_1(v) \phi_2(u) e^{-v^2/2s}e^{ - \frac{(u+v)^2}{2(t-s)}} dudv.
\end{align}

\subsubsection{Estimate of $A_2 (n)$} \label{Sec:3.2.2} 
We decompose $A_2 (n)$ as 
\begin{align*}
&\sum_{y=0}^{+\infty} \sum_{k\leq [ns]} \sum_{l\geq 0} \bE_x \Bigg[ \phi_1 \Big( \frac{X([ns])}{\sigma \sqrt{n}}\Big) \phi_2 \Big( \frac{X([nt])}{\sigma \sqrt{n}}\Big); \\
& \hskip 2cm r_l = k, X(k)= y,y + \xi_{k+1}\geq 0, \ldots, y+ \xi_{k+1} + \cdots + \xi_{[nt]}\geq 0 \Bigg]\\
=&\sum_{y=0}^{+\infty} \sum_{k\leq [ns]}  \bE_x \Bigg[ \phi_1 \Big( \frac{y+ \xi_{k+1}+\cdots + \xi_{[ns]}}{\sigma \sqrt{n}}\Big) \phi_2 \Big( \frac{y + \xi_{k+1} + \cdots + \xi_{[nt]}}{\sigma \sqrt{n}}\Big); \\
& \hskip 2cm y + \xi_{k+1} \geq 0, \ldots, y+ \xi_{k+1} + \cdots + \xi_{[nt]} \geq 0 \Bigg]\\
&\hskip 2cm  \times \sum_{l\geq 0}\pr_x[r_l = k, X(k)= y].
\end{align*}
Since $(\xi_k)$ is a i.i.d. sequence, 
\begin{align*}
A_2(n)  =&\sum_{y=0}^{+\infty} \sum_{k\leq [ns]} \Sigma_k(x, y)  \bE \Bigg[ \phi_1 \Big( \frac{y+ S([ns]-k)}{\sigma \sqrt{n}}\Big) \phi_2 \Big( \frac{y + S([nt]-k)}{\sigma \sqrt{n}}\Big);\\
&\hskip 8cm \tau^S(y) > [nt]-k \Bigg]. 
\end{align*}
For $u \in (0,s]$, we denote 
\begin{align*} g_n(u) &= n\sum_{y=0}^{+\infty}    \Sigma_{[nu]}(x, y)  \bE \Bigg[ \phi_1 \Big( \frac{y+ S([ns]-[nu])}{\sigma \sqrt{n}}\Big) \phi_2 \Big( \frac{y + S([nt]-[nu])}{\sigma \sqrt{n}}\Big); \\
 &\hskip 8.4cm \tau^S(y) > [nt]-[nu] \Bigg].
\end{align*} 
Now, let us compute  the pointwise limit on $(0, s]$ of the sequence $(g_n)_{n \geq 1}$. We write $g_n(u) $ as
\begin{align*}
g_n(u) = &n\sum_{y=0}^{+\infty}    \Sigma_{[nu]}(x, y) \\
& \hskip 0.5cm \times \bE \Bigg[ \phi_1 \Big( \frac{y+ S([ns]-[nu])}{\sigma \sqrt{n}}\Big) \phi_2 \Big( \frac{y + S([nt]-[nu])}{\sigma \sqrt{n}}\Big)\ \Big| \tau^S(y) > [nt]-[nu] \Bigg] \\
&\hskip 0.5cm \times \pr_y \left[\tau^S(y) > [nt]-[nu] \right].
\end{align*}
We set 
\[
a_n(x, y)  = n    \Sigma_{[nu]}(x, y)    \pr_y \left[\tau^S(y) > [nt]-[nu] \right],
\]
and
\[
b_n(y)  = \bE \Big[ \phi_1 \Big( \frac{y+ S([ns]-[nu])}{\sigma \sqrt{n}}\Big) \phi_2 \Big( \frac{y + S([nt]-[nu])}{\sigma \sqrt{n}}\Big)\ \Big| \tau^S(y) > [nt]-[nu] \Big].
\]
Note that $\displaystyle \sum_{y=0}^{+\infty} a_n(y) = \widehat{\Sigma}_{[nu]}(x,t,u)$.  
Since $\phi_1, \phi_2$ are bounded and  continuous on $\mathbb R$, it follows from Theorem 3.2 in \cite{Bol} and Theorems 2.23 and 3.4 in \cite{Iglehart74} that %\footnote{In  \cite{Iglehart74} the author needed the third order moment of the increment is finite; in fact, it only requires finite second moment  \cite{Bol}. }
\begin{align*}
\lim_{n \to +\infty} b_n(y)
&=\lim_{n \to +\infty} \bE\Big[ \phi_1\left( \frac{y+ S(  [ns]-[nu])}{\sigma\sqrt{[nt]-[nu]}} \frac{\sqrt{[nt]-[nu]}}{\sqrt{n}} \right) \\
&  \qquad \times \phi_2 \left( \frac{y+ S([nt]-[nu])}{\sigma\sqrt{[nt]-[nu]}} \frac{\sqrt{[nt]-[nu]}}{\sqrt{n}}\right)\Big|    \tau^S(y) > [nt]-[nu]\Big]\\
%	&= \int_0^{+\infty} \int_0^{+\infty} \phi_1(y\sqrt{t-u})\phi_2(z\sqrt{t-u})p(0,0;\frac{s-u}{t-u}, y) p(1-\frac{s-u}{t-u},y;1,z)dydz\\
&=  \int_0^{+\infty} \int_0^{+\infty}\phi_1(y\sqrt{t-u})\phi_2(z\sqrt{t-u}) 
\Big( \frac{t-u}{s-u}\Big)^{{3/2}}
ye^{-\frac{  t-u }{2(s-u)}y^2} \\
& \qquad  \qquad \qquad \times \frac{e^{-{1\over 2}  {t-u\over t-s}(z-y)^2} - e^{-{1\over 2}  {t-u\over t-s}(z+y)^2}}{\sqrt{2\pi\Big(1 - \frac{s-u}{t-u}\Big)}}dy dz\\
&=  {1\over \sqrt{2\pi(t-s)}} \int_0^{+\infty} \int_0^{+\infty}\phi_1(y')\phi_2(z') 
\frac{\sqrt{t-u}}{ (s-u)^{{3/2}}}
y' 
e^{-\frac{y'^2}{2(s-u)}}\\
& \qquad  \qquad \qquad \times \left(e^{-\frac{(z'-y')^2}{2(t-s)}} - e^{-\frac{(z'+y')^2}{2(t-s)}} \right)dy' dz'.
\end{align*}
Again, we can use the argument in the proof of Lemma \ref{1-dim} to show that the sequence $(g_n)$ converges point wise to $g$ with  
\begin{align*} 
g (u) =  &{1\over\pi^{{3/2}} \sqrt{2  (t-s)}}  {1\over \sqrt{u(s-u)^3}}
\\
&\qquad  \times \int_0^{+\infty} \int_0^{+\infty}\phi_1(y')\phi_2(z') 
y'e^{- {y'^2\over 2(s-u)}}
\left(e^{-{(z'-y')^2\over 2(t-s)}} - e^{- {(z'+y')^2\over 2(t-s)} }\right)dy' dz',
\end{align*}
and $(g_n)$ is also dominated by a function which is integrable on $[0,s]$ with respect to the Lebesgue measure. 
Lebesgue's dominated convergence theorem yields
\begin{align*}
&\lim_{n\to +\infty} A_2 (n)  
 = \lim_{n\to +\infty} \frac{1}{n} \sum_{k\leq [ns]}g_n(k/n) = \int_0^{s} g (u) du   \\
&\quad= {1\over\pi^{{3/2}} \sqrt{2\ (t-s)}} \int_0^s du  \int_0^{+\infty} dy'\int_0^{+\infty}dz'   \\
&\hskip 2cm \times 
\phi_1(y')\phi_2(z') 
\frac{e^{-\frac{y'^2}{2(s-u)}}}{\sqrt{u(s-u)^3}}
\frac{y'}{\sqrt{2\pi(t-s)}}
\left(e^{-\frac{(z'-y')^2}{2(t-s)}} - e^{-\frac{(z'+y')^2}{2(t-s)}} \right)   \\
&\quad= {1\over\pi^{{3/2}}   s\sqrt{2\ (t-s)}}  \int_0^{+\infty} dy'\int_0^{+\infty}dz'\phi_1(y')\phi_2(z') 
\left(e^{-\frac{(z'-y')^2}{2(t-s)}} - e^{-\frac{(z'+y')^2}{2(t-s)}} \right)   \\
& \hskip 4cm \times 
\left(\int_0^1 {y' \over \sqrt{v(1-v)^3}}   e^{-\frac{y'^2}{2s(1-v)}}dv\right) \\
\end{align*}
\begin{align}
&\quad=   {1\over\pi \sqrt{  s (t-s)}}  \int_0^{+\infty} dy' \int_0^{+\infty} dz' \phi_1(y')\phi_2(z')  
e^{-y'^2/2s} 
\left(e^{-\frac{(z'-y')^2}{2(t-s)}} - e^{-\frac{(z'+y')^2}{2(t-s)}} \right).
\label{lim:A2} 
%	&=\bE[\phi_1(|W_s|)\phi_2(|W_t|)] . 
\end{align}

\subsubsection{Conclusion}  
Combining \eqref{lim:A1} and \eqref{lim:A2}, we may write
\begin{align*}
 \lim_{n \to+ \infty}& \bE\left[ \phi_1\left( \frac{X({[ns]})}{\sigma\sqrt{n}}\right) \phi_2 \left( \frac{X([nt])}{\sigma\sqrt{n}}\right)\right] 
\\
 &={1\over\pi \sqrt{  s (t-s)}}  \int_0^{+\infty} dy' \int_0^{+\infty} dz' \phi_1(y')\phi_2(z')  
e^{-y'^2/2s} 
\left(e^{-\frac{(z'-y')^2}{2(t-s)}} + e^{-\frac{(z'+y')^2}{2(t-s)}} \right)
\\
&  = \bE[\phi_1(|B_s|)\phi_2(|B_t|)].
\end{align*}
Using a similar estimate as the one in  \eqref{lagtime}, we get 
$$\lim_{n \to+ \infty} \bE\left[ \phi_1\left(X_n(s)\right) \phi_2 \left(X_n(t)\right)\right]  = \bE[\phi_1(|B_s|)\phi_2(|B_t|)],$$
which concludes the convergence of $(X_n)$ in two-dimensional marginal distribution to a reflected Brownian motion. 
  
 \subsection{Finite dimensional distributions} 		The convergence of $d$-dimensional marginal distributions of $(X_n(t))_{n\geq 1}$ for any $d \geq 2$ may be done by induction on  $d$.  Let us fix $n\geq 1, d\geq 3$, then reals   $ 0<s_1<\cdots < s_d$  and $\phi_1, \ldots, \phi_d$   bounded and Lipschitz continuous   real valued functions defined on $\mathbb{R}$. 
%		For $j=1, \ldots, d-1$, let $\kappa_{n, j}$ denote the first reflection time after $[ns_j]$, i.e., 
%		$\kappa_{n, j}  = \min\{ k> [ns_j]: X(k-1) + \xi_k < 0\}$.
%		
%		We decompose $\displaystyle \bE_x \left [\prod_{i=1}^d \phi_i \left( \frac{X([ns_i])}{\sigma \sqrt{n}}\right) \right]$ as
%		\begin{align*}
%		\bE_x \left [\prod_{i=1}^d \phi_i \left( \frac{X([ns_i])}{\sigma \sqrt{n}}\right) \right]
%		= &\sum_{j=1}^{d-1}  \sum_{k=[ns_j]+1}^{[ns_{j+1}]} \bE_x \left [\prod_{i=1}^d \phi_i \left( \frac{X([ns_i])}{\sigma \sqrt{n}}\right); \kappa_{n, j}   = k \right ]\\
%		& \qquad \qquad \qquad + \bE_x \left [\prod_{i=1}^d \phi_i \left( \frac{X([ns_i])}{\sigma \sqrt{n}}\right) ; \kappa_{n, j}   > [ns_{j+1}] \right].
%		\end{align*}
%		Then we can deal with the terms   $\bE_x \left [\prod_{i=1}^d \phi_i \left( \frac{X([ns_i])}{\sigma \sqrt{n}}\right); \kappa_{n, j}   = k \right ]$ and  \\ $\bE_x \left [\prod_{i=1}^d \phi_i \left( \frac{X([ns_i])}{\sigma \sqrt{n}}\right) ; \kappa_{n, j}   > [ns_{j+1}] \right]$ in the same ways as we do in the previous subsection for $A_1(n)$ and $A_2(n)$, respectively. 
 
 Let $\kappa$ denote the first reflection time after $[ns_1]$, i.e., 
 $\kappa = \kappa(n,s_1) = \min\{ k> [ns_1]: X(k-1) + \xi_k < 0\}$.
 We decompose $\displaystyle \bE_x \left [\prod_{i=1}^d \phi_i \left( \frac{X([ns_i])}{\sigma \sqrt{n}}\right) \right]$ as 
 \begin{align*}
% 	\bE_x \left [\prod_{i=1}^d \phi_i \left( \frac{X([ns_i])}{\sigma \sqrt{n}}\right) \right]
% 	= &
	\sum_{j=1}^{d-1}  \sum_{k=[ns_j]+1}^{[ns_{j+1}]} \bE_x \left [\prod_{i=1}^d \phi_i \left( \frac{X([ns_i])}{\sigma \sqrt{n}}\right); \kappa = k \right ]
 	  + \bE_x \left [\prod_{i=1}^d \phi_i \left( \frac{X([ns_i])}{\sigma \sqrt{n}}\right) ; \kappa > [ns_d] \right].
 \end{align*}
 Then we can deal with the terms   $$
 \bE_x \left [\prod_{i=1}^d \phi_i \left( \frac{X([ns_i])}{\sigma \sqrt{n}}\right); \kappa = k \right ]\quad {\rm and} \quad  \bE_x \left [\prod_{i=1}^d \phi_i \left( \frac{X([ns_i])}{\sigma \sqrt{n}}\right) ; \kappa > [ns_d] \right]
 $$ in the same ways as we do for $A_1$ and $A_2$, respectively. 
 
 More precisely, 
 for each $1\leq j \leq d-1$ and $k \in \{[ns_j]+1, \ldots,  [ns_{j+1}]\}$, we write 
 \begin{align*}
& 
%\sum_{k=[ns_j]+1}^{[ns_{j+1}]}
	\bE_x \left [\prod_{i=1}^d \phi_i \left( \frac{X([ns_i])}{\sigma \sqrt{n}}\right); \kappa = k \right ]\\
&= \sum_{k_1=0}^{[ns_1] -1} 
%\sum_{k=[ns_j]+1}^{[ns_{j+1}]} 
\sum_{l\geq 0} \sum_{y\geq 1} \sum_{z\geq 1} \sum_{w\geq 0} 
	\bE_x \left [\prod_{i=1}^d \phi_i \left( \frac{X([ns_i])}{\sigma \sqrt{n}}\right); r_l=k_1, X(k_1)=z, \right.\\
	 &\left.  z+\xi_{k_1+1} \geq 0, \ldots, z + \xi_{k_1+1}+\cdots+ \xi_{k - 2}\geq 0, z +  \xi_{k_1+1}+\cdots+ \xi_{k - 1}= w, w + \xi_{k} = -y \right ]\\
&= \sum_{k_1=0}^{[ns_1] -1} 
%\sum_{k=[ns_j]+1}^{[ns_{j+1}]} 
\sum_{l\geq 0} \sum_{y\geq 1} \sum_{z\geq 1} \sum_{w\geq 0} 
\bE_x \left [\prod_{i_1=1}^j  \phi_{ i_1}\left( \frac{z+ \xi_{k_1+1}+\cdots + \xi_{[ns_j]}}{\sigma \sqrt{n}}\right) \right. \\
& \hspace{1.5cm} \times  
 \prod_{i_2=j+1}^d  \phi_{i_2} \left( \frac{y+ \xi_{k+1}+\cdots + \xi_{[ns_j]}}{\sigma \sqrt{n}}\right);
r_l=k_1, X(k_1)=z,   z+\xi_{k_1+1} \geq 0, \ldots, \\
&   \hspace{1.5cm}\left.  z + \xi_{k_1+1}+\cdots+ \xi_{k - 2}\geq 0, z +  \xi_{k_1+1}+\cdots+ \xi_{k - 1}= w,   w + \xi_{k} = -y \right ]\\
&= \sum_{k_1=0}^{[ns_1] -1} \sum_{z\geq 1} \Sigma_{k_1}(x,z)  
%\sum_{k=[ns_j]+1}^{[ns_{j+1}]}
  \sum_{y\geq 1}  \sum_{w\geq 0} \bE_y \left[\prod_{i_2=j+1}^d  \phi_{i_2} \left( \frac{X([ns_j]-k_2)}{\sigma \sqrt{n}}\right)\right]\\
& \hspace{1.5cm} \times  \bE\left [\prod_{i_1=1}^j  \phi_{i_1} \left( \frac{z+ S([ns_j]-k_1)}{\sigma \sqrt{n}}\right) \Big |   
\tau^S(z) > k-k_1 -1, z + S(k-k_1 - 1) = w \right] \\
& \hspace{1.5cm} \times \mathbb P [ \tau^S(z) > k - k_1 - 1, z + S(k-k_1 - 1) = w] \mathbb P[\xi_1 = -w-y].
\end{align*}  
Now we can use the induction hypothesis and  Corollary 2.5   in \cite{CC} to deal with the first and the second expectations.

\subsection{Tightness}
Recall that the modulus of continuity of a function $f:[0,1] \to \mathbb{R}$ is defined by 
$$w_f(\delta) = \sup_{ t,s \in [0,1], |t-s| < \delta}|f(t) - f(s)|.$$
It is clear that $w_{X}(\delta) \leq w_{S}(\delta)$. Using Theorem 7.3 in \cite{Billingsley}, the tightness of $X$ follows directly from the one of the classical random walk $(S(n))_{n \geq 0}$. 
We achieve  the proof of  Theorem \ref{meander}, 
applying Theorem 7.1 in \cite{Billingsley}. 

%%%%%%%%%%%%%%%%%%
\section{Auxiliary proofs}\label{sectionAuxiliaryproofs}
{\bf Proof of Lemma \ref{lem:M}.} 
By setting $h_n(y) =  \sqrt{n} \pr_y[r_1 > n]$, the Markov property yields 
\begin{align*}
\widehat{\Sigma}_n(x,t,s) 
&= n\sum_{l\geq 0} \bE_x \left[\pr_{X(r_l)}[r_1 \circ \theta^{r_l} > [nt] - [ns]];r_l = [ns]\right]\\
&={ \sqrt{n}\over \sqrt{[nt]-[ns]}} \ \sqrt{n}\sum_{l\geq 0} \bE_x \left[h_{[nt]-[ns]}(X(r_l));r_l = [ns]\right]\\
&= {1+o(n)\over \sqrt{s(t-s)}}  \sqrt{[ns]} T_{[ns]}(h_{[nt]-[ns]})(x).
\end{align*} 
Let us prove that  $\sqrt{[ns]} T_{[ns]}(h_{[nt]-[ns]})(x)\rightarrow \frac{1}{ \pi }   $ as $n\to +\infty$. Indeed,
\begin{align*}
\left \vert \sqrt{[ns]}T_{[ns]}(h_{[nt]-[ns]})(x) - \frac{1}{ \pi }\right\vert  &\leq B_1(n) + B_2(n), 
\end{align*}
with
\begin{align*} 
B_1(n)  &= \left \vert \sqrt{[ns]}T_{[ns]}(h_{[nt]-[ns]})(x) -\frac{1}{ \pi  \nu(h) } \nu (h_{[nt]-[ns]}) \right\vert, \quad \text{and} \\ 
B_2(n)  &= \frac{1}{ \pi  \nu(h) } \left\vert \nu (h_{[nt]-[ns]})-  \nu(h) \right\vert.
\end{align*} 
By Lemma \ref{LemA}, it holds $0 \leq h_n(y) \leq  C_1 h(y)$, with  $h(y)= O(y)$, so that   the sequence $( h_n)_{n \geq 1}$ is bounded in $\mathcal B_\alpha$. Thus,    Corollary \ref{G4.1} yields
\[
B_1(n)\leq(1+x)\left \vert \sqrt{[ns]}T_{[ns]}-\frac{1}{\pi  \nu(h) } \Pi  \right\vert_ \alpha \vert h_{[nt]-[ns]} \vert_\alpha \longrightarrow 0 \quad \text{as} \quad n\to +\infty.
\]
Similarly, by Lemma \ref{LemA} and the dominated convergence theorem,
\[
\lim_{n \to +\infty} \left\vert \nu (h_{[nt]-[ns]})-  \nu(h) \right\vert=0 , 
\]
so that $B_2(n) \longrightarrow 0$ as $n \to +\infty$.
\rightline{$\Box$}

\noindent {\bf Proof of Lemma \ref{lem:M'}.} 
By setting $\tilde{h}_n(y) =   n^{3/2}  \pr_y[r_1 = n]$, the Markov property yields 
%\textcolor{blue}{We define $\tilde{h}_n(y) = \frac{2n^{{3/2}}}{c_1} \pr_y[r_1 = n]$. Then we have 
\begin{align*}
\widetilde{\Sigma}_n(x,s,t)
&= n^2\sum_{l\geq 0} \bE_x \left[\pr_{X(r_l)}[r_1 \circ \theta^{r_l} = [nt] - [ns]];r_l = [ns]\right]\\
&={ n^{3/2}\over ([nt]-[ns])^{3/2}} \ \sqrt{n}\sum_{l\geq 0} \bE_x \left[\tilde h_{[nt]-[ns]}(X(r_l));r_l = [ns]\right]\\
&= {1+o(n)\over \sqrt{s}(t-s)^{3/2}}  \sqrt{[ns]} T_{[ns]}(\tilde h_{[nt]-[ns]})(x).
\end{align*} 
By Corollary \ref{cororenewal}, it holds $0 \leq \tilde h_n(y) \leq  C_3 h(y)$, with  $h(y)= O(y)$, so that   the sequence $( \tilde h_n)_{n \geq 1}$ is bounded in $\mathcal B_\alpha$.  We conclude as  above  to prove   Lemma  \ref{lem:M}.

\rightline{$\Box$}

\end{document}